\documentclass[10pt]{amsart}

\makeatletter
\def\blfootnote{\gdef\@thefnmark{}\@footnotetext}
\makeatother

\usepackage{epsfig}
\usepackage{graphics}
\usepackage{dcpic, pictexwd}
\usepackage{float}

\usepackage{mathtools}
\usepackage{amsmath,amssymb}
\usepackage{subcaption}
\usepackage{adjustbox}
\usepackage{xcolor}
\theoremstyle{plain}
\usepackage{amsthm}
\usepackage[square,numbers]{natbib}
\usepackage{thmtools}

\newtheorem*{theorem*}{Theorem}
\newtheorem{thmA}{Theorem}

\newtheorem{thmB}{Theorem}

\newtheorem{thmC}{Theorem}

\newtheorem{thmD}{Theorem}

\newtheorem{thmE}{Theorem}

\newtheorem{thmF}{Theorem}

\newtheorem{theorem}{Theorem}[section]

\newtheorem{lemma}[theorem]{Lemma}

\theoremstyle{remark}
\newtheorem{remark}[theorem]{Remark}
\newtheorem*{remark*}{Remark}

\theoremstyle{definition}
\newtheorem{definition}[theorem]{Definition}

\newtheorem{conj}{Conjecture}

\definecolor{darkgreen}{rgb}{0.0, 0.5, 0.0}

\usepackage[colorlinks=true,
                    linkcolor=blue,
                    urlcolor=blue,
                    citecolor=blue,
                    anchorcolor=blue]{hyperref}

 \def\Z{{\mathbb{Z}}}

\def\mod{{\rm Map}}
\def\pmod{{\rm PMap}}

\def\Ends{{\rm Ends}}

 \def\Sym{{\rm Sym}}
 \def\conj#1#2{#1^{#2}}
 \DeclarePairedDelimiter{\set}{\{}{\}}

\begin{document}

\blfootnote{\textup{2000} \textit{Mathematics Subject Classification}:
57M07, 20F05, 20F38}
\newenvironment{prooff}{\medskip \par \noindent {\it Proof}\ }{\hfill
$\square$ \medskip \par}
    \def\sqr#1#2{{\vcenter{\hrule height.#2pt
        \hbox{\vrule width.#2pt height#1pt \kern#1pt
            \vrule width.#2pt}\hrule height.#2pt}}}
    \def\square{\mathchoice\sqr67\sqr67\sqr{2.1}6\sqr{1.5}6}
\def\pf#1{\medskip \par \noindent {\it #1.}\ }
\def\endpf{\hfill $\square$ \medskip \par}
\def\demo#1{\medskip \par \noindent {\it #1.}\ }
\def\enddemo{\medskip \par}
\def\qed{~\hfill$\square$}

 \title[Small Torsion Topological Generators for Big Mapping Class Groups] {Small Torsion Topological Generators for Big Mapping Class Groups}

\author[T{\"{u}}l\.{i}n Altun{\"{o}}z, Celal Can Bellek, Em{\.{i}}r G{\"{u}}l,       Mehmetc\.{i}k Pamuk, and O\u{g}uz Y{\i}ld{\i}z ]{T{\"{u}}l\.{i}n Altun{\"{o}}z, Celal Can Bellek, Em{\.{i}}r G{\"{u}}l,      Mehmetc\.{i}k Pamuk, and O\u{g}uz Y{\i}ld{\i}z}
\address{Faculty of Engineering, Ba\c{s}kent University, Ankara, Turkey} 
\email{tulinaltunoz@baskent.edu.tr} 
\address{Department of Mathematics, Middle East Technical University,
 Ankara, Turkey}
\email{celal.bellek@metu.edu.tr}
\address{Department of Mathematics, Middle East Technical University,
 Ankara, Turkey}
\email{gul.emir@metu.edu.tr} 
\address{Department of Mathematics, Middle East Technical University,
 Ankara, Turkey}
 \email{mpamuk@metu.edu.tr}
 \address{Department of Mathematics, Middle East Technical University,
 Ankara, Turkey}
  \email{oguzyildiz16@gmail.com}

\begin{abstract}  

Let $S(n)$, for $n \in \mathbb{N}$, be the infinite-type surface of infinite genus
with $n$ ends, each accumulated by genus.
Although the mapping class groups of these surfaces are not countably generated,
they are Polish groups and hence admit a countable topological generating set.
We study minimal topological generating sets for $\mathrm{Map}(S(n))$ consisting
entirely of torsion elements, with special attention to involutions. In particular, we prove that $\mathrm{Map}(S(n))$ is topologically
generated by \emph{four} involutions for all $n \geq 16$, and by \emph{three}
involutions for the Loch Ness Monster surface ($n = 1$) and the Jacob's Ladder
surface ($n = 2$). We also establish that for even
$n \geq 8$, $\mathrm{Map}(S(n))$ is topologically generated by four torsion
elements of order $n$. For odd $n \geq 8$, it is topologically generated by three
torsion elements of order $n$ and one torsion element of order $n - 1$.

\end{abstract}

\maketitle
  \setcounter{secnumdepth}{2}
 \setcounter{section}{0}

The mapping class group of a surface $S$, denoted $\mathrm{Map}(S)$, is the group
of isotopy classes of orientation-preserving self-homeomorphisms of $S$ that fix
the boundary components (if any) pointwise and permute the ends setwise. It is a fundamental object in low-dimensional topology, encoding the symmetries of
the surface. For surfaces of \emph{finite type} (those with a finitely generated fundamental
group), the structure of this group is well understood: a classical theorem shows that it is finitely generated by Dehn twists~\cite{Dehn1987,Humphries1979,Lickorish1964}, and Wajnryb later showed that two generators suffice~\cite{Wajnyrb1996}, a result sharpened further by Korkmaz~\cite{Korkmaz2004}.

A particularly fruitful line of research has been the generation of mapping class groups by torsion elements, especially involutions. The theoretical motivation for this dates back to Maclachlan\cite{Maclachlan1971}, who demonstrated that the generation of the mapping class group by torsion elements implies that the moduli space of algebraic curves is simply-connected. McCarthy and Papadopoulos showed that for genus $g \geq 3$, $\mathrm{Map}(S)$
is generated by infinitely many conjugates of a single involution~\cite{McCarthy1987}. Luo later demonstrated that a \emph{finite} collection of involutions suffices~\cite{Luo2000}, sparking a series of improvements. Brendle and Farb reduced the count to six involutions for $g \geq 3$~\cite{Brendle2004}, Kassabov subsequently brought this down to four for $g \geq 7$~\cite{Kassabov2003}, and Korkmaz and the fifth author reduced it further to three for $g \geq 6$~\cite{Korkmaz2020,Yildiz2020}. Similar minimal generation results utilizing torsion elements and involutions have also been successfully established for finite-type surfaces with punctures by Monden and the others. 
 
In recent years, attention has shifted to mapping class groups of
infinite-type surfaces, often called big mapping class groups. These groups are structurally richer: they are not finitely generated. However, when equipped with the compact-open topology they become Polish groups, and are hence topologically generated by a countable dense subgroup. Their topological generating sets must include not only Dehn twists but also homeomorphisms of infinite support called handle shifts, as shown by Patel and Vlamis~\cite{Patel2018}.  The topological normal generation problem for these groups was recently
studied by Baik~\cite{Baik2024}, who proved that, for surfaces with
countable end space, topological normal generation is equivalent to
unique self-similarity.  The authors recently proved that $\mathrm{Map}(S(n))$ is topologically generated by at most four elements~\cite{mingen}.
 
This paper is a companion to~\cite{mingen} and uses similar methods to study small topological generating sets consisting of torsion elements for the mapping class groups of the surfaces $S(n)$, surfaces with infinite genus and exactly $n$ ends, each accumulated by genus. The paper pursues two objectives: minimizing the number of \emph{involution} generators, and constructing topological generating sets of torsion elements of certain orders.
 
\medskip\noindent\textbf{Involution generators.}
For the surfaces $S(n)$, Huynh showed that $\mathrm{Map}(S(n))$ can be topologically generated by seven involutions for $n \geq 3$~\cite{huynh}. This was subsequently improved by the first, fourth, and fifth authors, who reduced the bound to six involutions for $n \geq 3$ and to five for $n \geq 6$~\cite{apyinvolution}.  More broadly, Malestein and Tao studied generation by involutions for
self-similar surfaces~\cite{MalesteinTao2024}, showing that the mapping class groups of uniformly self-similar surfaces are generated
by involutions and are in fact normally generated by a single involution. The present paper establishes further reductions to topological generation results (Theorems~A and B). Additionally, this paper establishes topological generating sets of involutions for the Loch Ness Monster surface and the Jacob's Ladder surface (Theorems~C and D).

\medskip\noindent\textbf{Torsion generators of higher order.}
Beyond involutions, one can ask for generating sets whose elements all have a prescribed finite order. For finite-type surfaces, Korkmaz~\cite{korkmaz2005} proved that $\mathrm{Map}(S)$ can be generated by exactly two torsion elements of order $4g+2$, when $g\geq 3$ and the surface has at most one boundary component. Building on this, the fifth author showed that $\mathrm{Map}(S)$ can be generated by two torsion elements of smaller orders, such as order $g$, for sufficiently large genus $g$~\cite{Yildiz2022}. We address the analogous question for big mapping class groups in Theorems~E and~F below, obtaining generating sets of four torsion elements of order $n$ for even $n \geq 8$, and of three elements of order $n$ plus one element of order $n-1$ for odd $n \geq 8$.

\subsection*{Main Results}\label{subsec:results}
 
The main contributions of this paper are as follows. For the involution generation results, the generating sets differ for even and odd
$n$, so we state these separately.
 
\begin{thmA}
  For any odd integer $n \geq 17$, the mapping class group $\mathrm{Map}(S(n))$
  is topologically generated by four involutions.
\end{thmA}
 
\begin{thmB}
  For any even integer $n \geq 16$, the mapping class group $\mathrm{Map}(S(n))$
  is topologically generated by four involutions.
\end{thmB}
 
\begin{thmC}
  The mapping class group of the Jacob's Ladder surface $S(2)$ is topologically
  generated by three involutions.
\end{thmC}
 
\begin{thmD}
  The mapping class group of the Loch Ness Monster surface $S(1)$ is topologically
  generated by three involutions.
\end{thmD}
 
For torsion elements, we establish:
 
\begin{thmE}
  For any even integer $n \geq 8$, the mapping class group $\mathrm{Map}(S(n))$
  is topologically generated by four torsion elements of order $n$.
\end{thmE}
 
\begin{thmF}
  For any odd integer $n \geq 8$, the mapping class group $\mathrm{Map}(S(n))$
  is topologically generated by three torsion elements of order $n$ and one torsion
  element of order $n - 1$.
\end{thmF}
 
\begin{remark*}\label{rem:minimality}
It is natural to ask whether the generating sets in Theorems~C and~D are optimal.
For purely algebraic reasons, \emph{two} involution generators are never enough:
any group generated by two involutions is a quotient of the infinite dihedral group
$D_\infty$ and is therefore virtually cyclic, whereas $\mathrm{Map}(S(n))$ contains
non-abelian free subgroups. Hence three is the minimum possible number of involution generators, confirming that Theorems~C and~D are sharp.

\end{remark*}
 
The paper is organized as follows. In Section~2 we briefly recall the necessary definitions and concepts from \cite{mingen}. Section~3 contains the proofs of the involution generation results Theorems~A through D, treating surfaces with $n \geq 16$ ends and the special cases of the Jacob's Ladder and Loch Ness Monster surfaces separately. Section~4 presents the higher-order torsion generation results, Theorems~E and~F, for $n \geq 8$.
 
\subsection*{Notation}\label{notation}
 
The notation follows standard conventions in mapping class group theory, with a few
specifics recorded here. The surface of infinite genus with $n$ ends, each accumulated by genus, is denoted $S(n)$, and its mapping class group $\mathrm{Map}(S(n))$. For brevity, a homeomorphism and its isotopy class are denoted by the same symbol. Group composition $f \circ g$ is written as $fg$. The right-handed Dehn twist about a simple closed curve $a$ (i.e.\ $t_a$) is written as the corresponding capital letter, e.g.\ $A$.
 
For the surfaces $S(n)$ with $n \geq 3$ ends, a \emph{double-index} system is used: specific Dehn twists are written $A_i^j$, $B_i^j$, $C_{i-1}^j$, corresponding to curves $a_i^j$, $b_i^j$, $c_{i-1}^j$. The lower index $i = 1,2,3,\ldots$ records the genus level (position along the infinite chain of genera within one end, as in Figure~\ref{fig:inf}), while the upper index $j = 1,\ldots,n$ records which of the $n$ accumulated ends the curve
lies near. The inverse of a mapping class $X$ is written $\overline{X}$. A handle shift is denoted $h$ or $h_{i,j}$. In the simplified notation used in some proofs (e.g.\ Theorem~\ref{thm:A} for $n \geq 17$), the lower index is suppressed so that $A_i^j$, $B_i^j$, $C_{i-1}^j$ are written $A_j$, $B_j$, $C_j$ respectively, with $i = 1,2,3,\ldots$ and $j = 1,\ldots,n$.
 
\begin{figure}[htbp]
  \centering
  \hspace*{-0.5cm}
  \scalebox{0.5}{\includegraphics{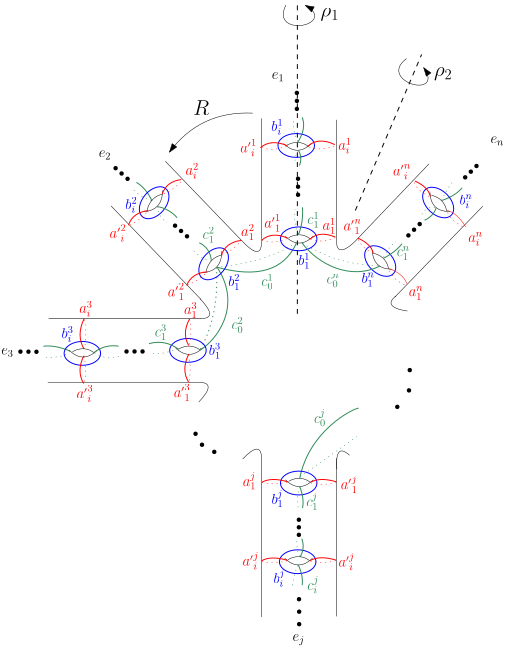}}
  \caption{The surface $S(n)$, showing the families of curves used for
  generating sets. The upper index labels the end, while the lower index labels the genus level.}
  \label{fig:inf}
\end{figure}
 
\noindent\textbf{Acknowledgements.}
This work is supported by the Scientific and Technological Research Council of
Turkey (T\"{U}B\.{I}TAK) [grant number 125F253].

\section{Preliminaries on infinite-type surfaces}
\subsection{Classification of infinite-type surfaces}
To classify surfaces of infinite type we use the \textit{space of ends} $\Ends(S)$, which records the distinct directions to infinity of the surface.  The construction begins with exiting sequences (nested connected open sets with compact boundary that eventually avoid every compact subset of $S$); $\Ends(S)$ is the set of equivalence classes of such sequences, equipped with the topology generated by the sets $U^*$, where $U\subset S$ is open with compact boundary and $U^*$ consists of those ends represented by exiting sequences eventually contained in $U$.  Intuitively, the space of ends describes how many different directions goes to infinity, how those directions relate to each other and whether those directions contain infinitely many genera or not. We say that an end is accumulated by genus if every element of the sequence defining the end contains infinitely many genera. The classification theorem for orientable infinite-type surfaces then asserts that two such surfaces are homeomorphic exactly when they have the same genus and number of boundary components and there exists a homeomorphism $\Ends(S_1)\cong\Ends(S_2)$. The definitions and conventions used here follow Aramayona-Vlamis~\cite{Aramayona_2020}.

\begin{theorem}
 Let $S_{1}$ and $S_{2}$ be two infinite-type surfaces. Let $g_1, g_2$ and $b_1, b_2$ be the number of genera and boundary components of these surfaces, respectively. Then, $S_1 \cong S_2$ if and only if $g_1 = g_2$, $b_1=b_2$, and there is a homeomorphism
    \begin{align*}
        \Ends(S_1) \rightarrow \Ends(S_2)
    \end{align*}
    that restricts to a homeomorphism between their respective subspaces consisting of ends accumulated by genus.
\end{theorem}

\subsection{Generating the big mapping class groups}

\begin{definition}
    \textit{The pure mapping class group}, denoted by $\pmod(S(n))$, is the subgroup of $\mod(S(n))$ such that it fixes $\Ends(S)$ pointwise.
\end{definition}

For the surfaces of infinite-type, $\mod(S(n))$ is not countably generated.  However, since it is a quotient of the group of orientation-preserving self-homeomorphisms of $S(n)$ (equipped with the compact-open topology), $\mod(S(n))$ inherits a topology. Because of this, $\mod(S(n))$ is a Polish group~\cite{Aramayona_2020}, meaning in particular that it is separable.  Therefore, $\mod(S(n))$ is topologically generated by a countable set.
We have the following exact sequence:
\begin{align*}
    1 \rightarrow \pmod(S(n)) \rightarrow \mod(S(n)) \rightarrow \Sym_n \rightarrow 1.
\end{align*}
Here, $\Sym_n$ is the symmetric group on $n$ letters and the last map is the projection defined by the action of a mapping class on the space of ends, which is the symmetric group on $n$ letters for $\mod(S(n))$. It follows that $\mod(S(n))$ is topologically generated by the generators of $\pmod(S(n))$ together with mapping classes whose image in $\Sym_n$ generate it.

\subsection*{Handle shifts}
The generators of these groups often include not only Dehn twists, but also homeomorphisms with infinite support called \textit{handle shifts} as shown by Patel and Vlamis~\cite{Patel2018}.

Following~\cite{tez}, we define the handle shift as follows:
Consider the surface $\mathbb{R} \times [-1,1]$ with disks of radius $1/4$ removed and a copy of $S_1^1$ attached along the boundaries of the removed disks at every point $(n,0)$ where $n \in \mathbb{Z}$. This surface is called the \textit{model surface of a handle shift} and denote it by $\Sigma$.

Note that $\Sigma$ is a surface with two ends accumulated by genus that correspond to $\pm \infty$ of $\mathbb{R}$ and two disjoint boundary components $\mathbb{R}\times\{-1\}$ and $\mathbb{R}\times\{1\}$. This means that we can embed $\Sigma$ to any infinite-type surface S with at least two ends accumulated by genus. We define $h: \Sigma \rightarrow \Sigma$ as
\begin{align*}
    h(x,y)=
    \begin{cases}
        (x+1,y) & \text{if }  y \in [-\dfrac{1}{2},\dfrac{1}{2}],\\
        (x+2-2y,y) & \text{if }  y \in [\dfrac{1}{2},1],\\
        (x+2+2y,y) & \text{if } y \in [-1, -\dfrac{1}{2}]
    \end{cases}
\end{align*}
on $\mathbb{R} \times [-1,1]$. This self-homeomorphism $h$ is called a \textit{handle-shift}.

\begin{figure}[htbp]
      \centering
      \includegraphics[width=0.60\textwidth]{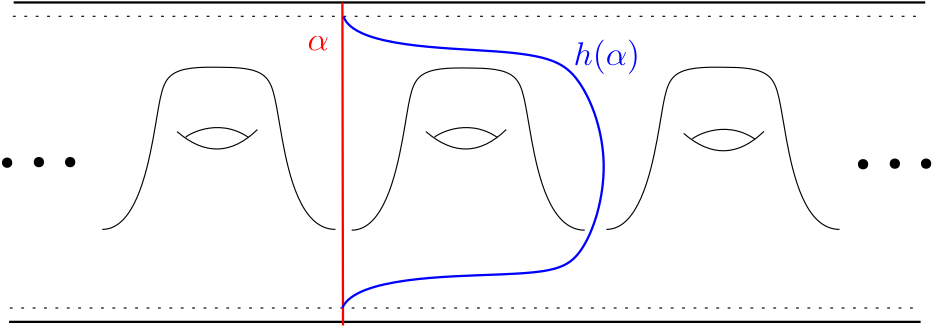}
      \caption{The action of a handle shift $h$ on a transverse curve $\alpha$. The model surface $\Sigma$ illustrates the shift of genera from one region to another}
      \label{fig:handleshift}
\end{figure}

Patel and Vlamis showed in their initial paper that for infinite-type surfaces with more than one end accumulated by genus, handle shifts and Dehn twists are required to topologically generate $\pmod(S)$ [Proposition 6.3, \citenum{Patel2018}]. Moreover Aramayona-Patel-Vlamis improved this result by proving that $\pmod(S)$ can be split as a semi-direct product of $\overline{\pmod_\mathrm{c}(S)}$ and a product of handle shifts \cite{arapatevla}. We state this result for the case relevant to us in this paper:
\begin{theorem}[{\cite[Corollary 6]{arapatevla}}]
    For $S(n)$,
    \[
    \pmod(S(n)) = \overline{\pmod_\mathrm{c}(S(n))} \rtimes \Z^{n-1}.
    \]
\end{theorem}

This result shows that any set that topologically generates $\overline{\pmod_\mathrm{c}(S(n))}$ and that contains $n-1$ handle shifts with different attracting and repelling ends topologically generates the entire pure mapping class group.

\subsection{Special infinite-type surfaces}
\subsection*{The Loch Ness Monster surface} The closed surface with one end accumulated by genus is called the Loch Ness Monster Surface.

\begin{figure}[htbp]
      \centering
      \includegraphics[width=0.40\textwidth]{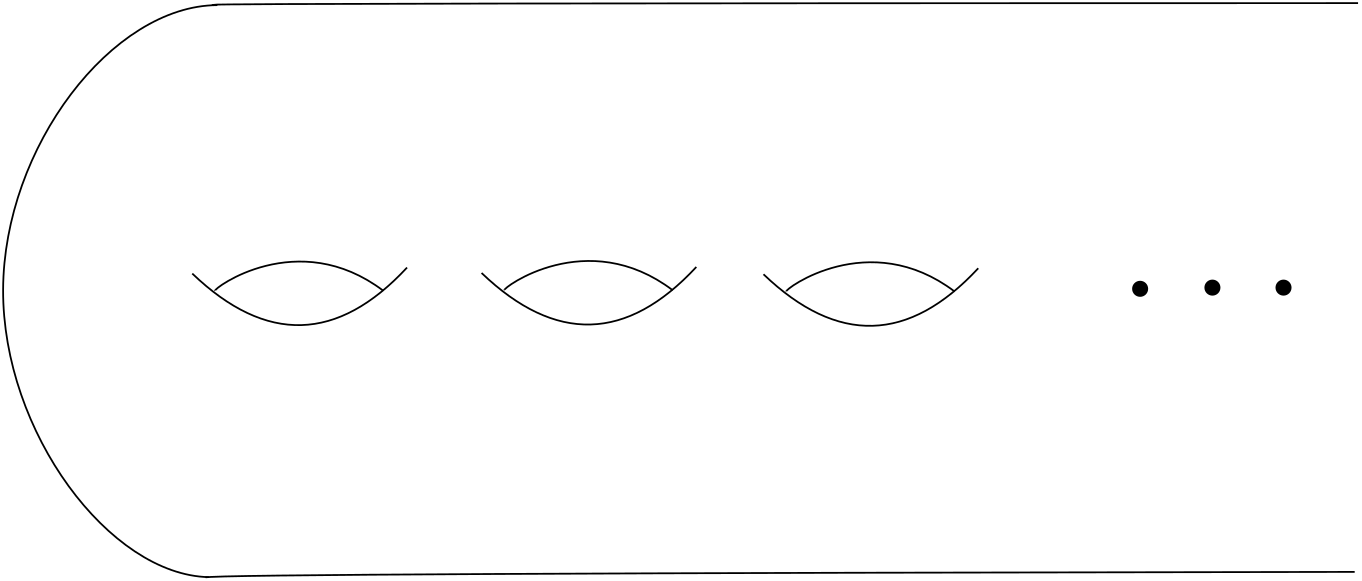}
      \caption{An embedding of the Loch Ness Monster surface, the infinite-genus surface with a single end.}
      \label{fig:model}
   \end{figure}

\subsection*{The Jacob's Ladder surface}\label{Jacob} The closed surface with two end accumulated by genus is called The Jacob's Ladder Surface.

\begin{figure}[htbp]
      \centering
      \includegraphics[width=0.40\textwidth]{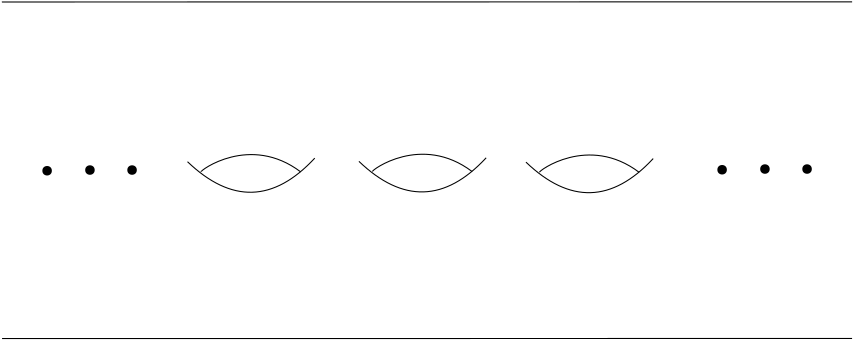}
      \caption{An embedding of the Jacob's Ladder surface, the infinite-genus surface with two ends}
      \label{Jacob's Ladder}
\end{figure}

\section{Involution Generators}

In this section, we provide the proofs for our main theorems. Our method is to show that our minimal generating sets contain certain elements that are enough to generate every Dehn twist and handle shifts in $\mod(S(n))$, then show that our generators are involutions.

\subsection{Surfaces with more than two ends}
Let $R$ be the counter-clockwise rotation of the ends of the surface $S(n)$. This homeomorphism rotates the surface by an angle of $\frac{2\pi}{n}$ radians. Let $\rho_1$ be a $\pi$ radian rotation through the first end and $\rho_2$ be a $\pi$ radian rotation between the ends numbered $\frac{n}{2}$ and $\frac{n}{2}+1$ if $n$ is even and rotation through the end numbered $\frac{n+1}{2}$ if $n$ is odd. The rotations $R,\rho_1$ and $\rho_2$ are depicted in Figure \ref{fig:inf}. We shall use the following theorem, which gives us a finite generating set for $\pmod_\mathrm{c}(S(n))$, to prove our results. 

\begin{theorem}
    [{\cite[Theorem~2.4]{apyinvolution}}]\label{thm:genlemma}
For $n\geq3$, the group topologically generated by the elements
\[
\lbrace \rho_1, \rho_2, A_1^{1}\overline{A_1^{2}},B_1^{1}\overline{B_1^{2}}, C_0^{1}\overline{C_0^{2}}, h_{1,2}\rbrace
\]
contains the Dehn twists $A_i^{j}$, $B_i^{j}$, $C_{i-1}^{j}$ for all $j=1,2,\ldots, n$ and for all $i=1, 2, 3, \ldots$. 
\end{theorem}

We use the following lemma to establish that certain elements in our generating sets are involutions.

\begin{lemma}[{\cite[Lemma~21]{Yildiz2022}}]\label{lemm8}
    If $R$ is an element of order $n$ in a group $G$, and if $x$ and $y$ are elements in $G$ satisfying $Rx\overline{R}=y$, then the order of $R x y^{-1}$ is also $n$.
\end{lemma}
Armed with Theorem~\ref{thm:genlemma} and Lemma~\ref{lemm8}, we now state and prove our main theorems. Note that we use the simplified notation for the statements and proofs of the theorems.

\begin{thmA}\label{thm:A}
For any odd integer $n \geq 17$, the mapping class group $\mod(S(n))$ is topologically generated by four involutions.
\end{thmA}
\begin{proof}
We define our generating set of four involutions as $I = \{\rho_1, \rho_2, \rho_3 F_1,\tau\}$, where
\[
F_1 = A_1 C_1 B_4 \overline{B_6} \mkern3mu\overline{C_8}\mkern3mu \overline{A_9'} h_{\frac{n+1}{2}+4, \frac{n+1}{2}+5},
\]
$\rho_3 = R^4\rho_1\overline{R^4}$, and $\tau$ is a homeomorphism that projects to $(1\,2)$ in $\Sym_n$. Such a $\tau$ always exists, see~\cite{apyinvolution,tez}. Note that $R = \rho_1\rho_2$ and $\rho_3 = {\rho_1}^{{R^4}}$, and are therefore in $G$. 
Let $G$ be the subgroup of $\mod(S(n))$ topologically generated by $I$. Since $\rho_1$ and $\rho_2$ are in $G$, $R$ and $\rho_3$ are also in $G$. It follows that $F_1$ is in $G$.

We start by showing that $A_1\overline{A_2}$, $B_1\overline{B_2}$, $C_1\overline{C_2}$ and $A_1'\overline{A_2'}$ are in $G$. We then show that the element $h_{\frac{n+1}{2}+4, \frac{n+1}{2}+5}$ is in $G$, and thus $h_{1,2}$ can be separated from the Dehn twists and is in $G$. By Theorem~\ref{thm:genlemma}, $G$ contains $A_i^j,B_i^j$ and $C_{i-1}^j$, for all $j=1,2,\ldots, n$ and for all $i=1, 2, 3, \ldots$, and thus $\overline{\pmod_\mathrm{c}(S(n))}< G$ by [\citenum{tez}, Proposition 6.1.10]. 

\begin{remark}
The following steps involve a series of conjugations that rely on the geometric configuration of the curves, in particular disjointness and intersection numbers, and the braid relation. For brevity, the detailed calculations demonstrating how terms cancel or transform are omitted. The key idea is that each conjugating element is chosen to commute with most terms of the target element, acting non-trivially only on specific components.
\end{remark}

\noindent Let 
\[
F_2 = \conj{F_1}{R^2} = A_3 C_3 B_6 \overline{B_8} \mkern3mu\overline{C_{10}}\mkern3mu \overline{A_{11}'} h_{\frac{n+1}{2}+6, \frac{n+1}{2}+7} \in G.
\]
Conjugating $F_1$ by $F_1F_2$ gives 
\[F_3 = \conj{F_1}{F_1 F_2} = A_1 C_1 C_3 \overline{B_6}\mkern3mu \overline{B_8}\mkern3mu \overline{A_9'} h_{\frac{n+1}{2}+4, \frac{n+1}{2}+5} \in G.
\]
Then let 
\[
F_4 = \conj{F_3}{R^2} = A_3 C_3 C_5 \overline{B_8}\mkern3mu \overline{B_{10}}\mkern3mu \overline{A_{11}'} h_{\frac{n+1}{2}+6, \frac{n+1}{2}+7} \in G,
\] and 
\[F_5 = \conj{F_3}{F_3\overline{F_4}} = A_1C_1C_3\overline{C_5}\mkern3mu \overline{B_8}\mkern3mu \overline{A_9'} h_{\frac{n+1}{2}+4, \frac{n+1}{2}+5}\in G.
\]
Note that the product $\overline{F_3} F_5 = B_6 \overline{C_5}$ is in $G$. Conjugating by powers of $R$, we see that $B_{i+1}\overline{C_i}$ is in $G$ for all applicable $i$.  Now, 
\[
F_6 = \conj{F_5}{{R^3}} = A_4C_4C_6\overline{C_8}\mkern3mu\overline{B_{11}}\mkern3mu \overline{A_{12}'} h_{\frac{n+1}{2}+7, \frac{n+1}{2}+8}\in G.
\] 
Conjugating $F_5$ by $F_5F_6$ yields
\[
F_7 = \conj{F_5}{F_5F_6} = A_1C_1C_3\overline{C_5}\mkern3mu \overline{C_8}\mkern3mu \overline{A_9'} h_{\frac{n+1}{2}+4, \frac{n+1}{2}+5} \in G.
\]

The product $\overline{F_5} F_7$ simplifies to $B_8 \overline{C_8}$, and is therefore in $G$. Its inverse, $C_8 \overline{B_8}$, is also in $G$. Conjugating this element by powers of $R$, it is clear that $C_i \overline{B_i}\in G$ for all $i$.

Combining these elements give
\[
 (B_1\overline{C_1})(C_1\overline{B_2})=B_1\overline{B_2}\in G,
\] 
and
\[
(C_1\overline{B_2})(B_2\overline{C_2})=C_1\overline{C_2} \in G.
\]

Since $B_i\overline{B_{i+1}}$ and $C_i\overline{C_{i+1}}$ are in $G$, we can simplify $F_1$. Notice that 
\[
(B_6\overline{B_5})(B_5\overline{B_4}) = B_6\overline{B_4} \in G,
\]
and
\[
(C_8\overline{C_7})(C_7\overline{C_6})\cdots(C_3\overline{C_2})(C_2\overline{C_1}) = C_8\overline{C_1}\in G.
\]
Therefore, 
\[
F_8= F_1(C_8\overline{C_1})(B_6\overline{B_4}) = A_1\overline{A'_9}h_{\frac{n+1}{2}+4,\frac{n+1}{2}+5} \in G.
\]
Conjugating $F_8$ by 
\[
F_8(B_1\overline{B_2}) =  A_1\overline{A'_9}h_{\frac{n+1}{2}+4,\frac{n+1}{2}+5}B_1\overline{B_2}
\]
gives 
\[
F_9 = \conj{F_8}{F_8(B_1\overline{B_2})} = B_1\overline{A_9'}h_{\frac{n+1}{2}+4,\frac{n+1}{2}+5} \in G,
\]
and the product $F_8\overline{F_{9}} = A_1\overline{B_1}$ is in $G$.  
Now,  
\[
(A_1\overline{B_1})(B_1\overline{B_2})(B_2\overline{A_2}) = A_1\overline{A_2} \in G.
\]
Note that we can apply the same procedure to show that $A'_1\overline{B_1}$ and $A'_1\overline{A'_2}$ are also in $G$.
Now note that 
\[
(B_1\overline{B_2})(B_2\overline{A'_2})(A'_2\overline{A'_3})\cdots(A'_7\overline{A_8'})(A_8'\overline{A_9'}) = B_1\overline{A_9'} \in G.
\]
Therefore, 
\[
(A_9'\overline{B_1})F_{9} = h_{\frac{n+1}{2}+4,\frac{n+1}{2}+5} \in G.
\]
It follows from conjugation by powers of $R$ that $h_{1,2}$ is in $G$.

Since $A_1\overline{A_2}$, $B_1\overline{B_2}$, $C_1\overline{C_2}$, $R$, and $h_{1,2}$ are in $G$, Theorem~\ref{thm:genlemma} and [\citenum{tez}, Proposition 6.1.10] implies that $G$ contains the closure of the compactly supported mapping class group. Since $R$ projects to $(1\,2\,\dots\,n)$ and $\tau$ projects to $(12)$ in $\Sym_n$, they generate it by a classical result. It follows that $G$ must be $\mod(S(n))$.

It remains to check that $\rho_3 F_1$ is an involution. A direct calculation shows that 
\begin{align*}
    \rho_3(A_1C_1B_4)\rho_3 = A_9'C_8B_6, \\
    \rho_3(\overline{B_6} \mkern3mu\overline{C_8}\mkern3mu \overline{A_9'}) \rho_3 = \overline{A_1}\mkern3mu \overline{C_1}\mkern3mu \overline{B_4},
\end{align*}
and 
\[
\rho_3 h_{\frac{n+1}{2}+4, \frac{n+1}{2}+5} \rho_3 = h_{\frac{n+1}{2}+5, \frac{n+1}{2}+4} = \overline{h_{\frac{n+1}{2}+4, \frac{n+1}{2}+5}}. 
\]
It immediately follows that
\[
(\rho_3 F_1)(\rho_3 F_1) = (\rho_3 F_1\rho_3) F_1 = \overline{h_{\frac{n+1}{2}+4, \frac{n+1}{2}+5}}(A_9'C_8B_6)(\overline{A_1}\mkern3mu \overline{C_1}\mkern3mu \overline{B_4})(B_4C_1A_1   )(\overline{B_6} \mkern3mu\overline{C_8}\mkern3mu \overline{A_9'}) h_{\frac{n+1}{2}+4, \frac{n+1}{2}+5} = 1,
\]
and that $\rho_3F_1$ is an involution.
\end{proof}

\begin{thmB}\label{thm:B}
For any even integer $n \geq 16$, the mapping class group $\mod(S(n))$ is topologically generated by four involutions.
\end{thmB}
\begin{proof}
Our strategy is the same as in the proof of Theorem~\ref{thm:A}. Let $G$ be the subgroup topologically generated by $I = \{\rho_1,\rho_2,\rho_3F_1,\tau\}$, where
\begin{align*}
    F_1 = A_1C_1B_4\overline{B_5}\mkern3mu\overline{C_7}\mkern3mu\overline{A_8'}h_{\frac{n}{2}+4,\frac{n}{2}+5},
\end{align*}
and $\rho_3 = R^4\rho_2\overline{R^4}$.
Let 
\begin{align*}
    F_2 = \conj{F_1}{R^2} = A_3C_3B_6\overline{B_7}\mkern3mu\overline{C_{9}}\mkern3mu\overline{A_{10}'}h_{\frac{n}{2}+6,\frac{n}{2}+7} \in G.
\end{align*}
Conjugating $F_1$ by $F_1F_2$, we obtain
\begin{align*}
    F_3 = \conj{F_1}{F_1F_2} = A_1C_1C_3\overline{B_5}\mkern3mu\overline{B_7}\mkern3mu\overline{A_8'}h_{\frac{n}{2}+4,\frac{n}{2}+5} \in G.
\end{align*}
Now, let 
\begin{align*}
    F_4 = \conj{F_3}{R^2} = A_3C_3C_5\overline{B_7}\mkern3mu\overline{B_{9}}\mkern3mu\overline{A_{10}'}h_{\frac{n}{2}+6,\frac{n}{2}+7} \in G.
\end{align*}
By conjugating $F_3$ with $F_3\overline{F_4}$ we obtain
\begin{align*}
    F_5 = \conj{F_3}{F_3\overline{F_4}} = A_1C_1C_3\overline{C_5}\mkern3mu\overline{B_7}\mkern3mu\overline{A_8'}h_{\frac{n}{2}+4,\frac{n}{2}+5} \in G.
\end{align*}
Note that $\overline{F_3}F_5 = B_5\overline{C_5} \in G$ which implies that $B_i\overline{C_i}$, and $C_i\overline{B_i}$ are both in $G$ through conjugations by powers of $R$. Then,
\begin{align*}
    &F_6 = F_1(C_7\overline{B_7}) = A_1C_1B_4\overline{B_5}\mkern3mu\overline{B_7}\mkern3mu\overline{A_8'}h_{\frac{n}{2}+4,\frac{n}{2}+5},\\
    &F_7=\conj{F_6}{R^2}=A_3C_3B_6\overline{B_7}\mkern3mu\overline{B_9}\mkern3mu\overline{A_{10}'}h_{\frac{n}{2}+6,\frac{n}{2}+7},\\
    &F_8=\conj{F_6}{F_6F_7}=A_1C_1C_3\overline{B_5}\mkern3mu\overline{B_7}\mkern3mu\overline{A_8'}h_{\frac{n}{2}+4,\frac{n}{2}+5},
\end{align*}
Therefore, $F_6\overline{F_8}=B_4\overline{C_3}$ is in $G$. It follows that $B_{i+1}\overline{C_i}$ is in $G$ for all $i\in \{1,\dots,n\}$ (note that $B_{n+1}=B_1$).

Thus, similar to the previous theorem, we can combine these elements so that $B_1\overline{B_2}$ and $C_1\overline{C_2}$ is in $G$. The rest of the proof follows the same steps as in the proof of Theorem~\ref{thm:A}.

\end{proof}

\subsection{The Jacob's Ladder Surface}
In this section, we focus on the Jacob's Ladder surface, the infinite-genus surface with two ends, both accumulated by genus. We will show that the topological generating set established in~\cite{mingen} can be so that each generator is an involution. We  use the model depicted in Figure~\ref{fig:curves}.

\begin{figure}[htbp]
      \centering
      \includegraphics[width=0.50\textwidth]{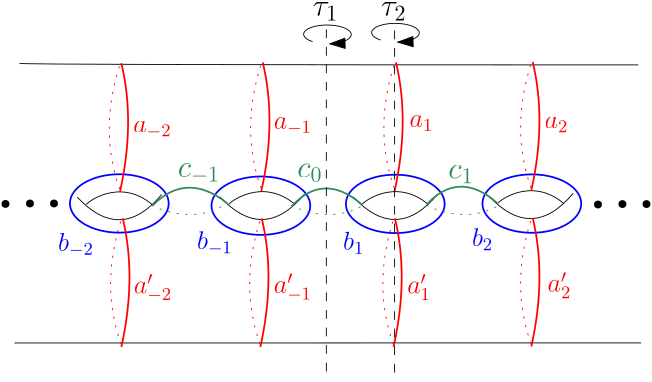}
      \caption{ A model for the Jacob's Ladder surface, showing the indexed curves and the rotations 
      $\tau_1$ and $\tau_2$.}
      \label{fig:curves}
\end{figure}
Observe that
\begin{align*}
    H= \tau_2\tau_1
\end{align*}
is a handle shift whose attracting end is $+\infty$ and repelling end is $-\infty$.

Throughout this subsection, the Jacob's Ladder surface is denoted by $S$.

\begin{theorem}[{\cite[Theorem 3.10]{mingen}}]\label{lem:jaclem}
    Let $S$ be the Jacob's Ladder surface. Then $\mathrm{Map}(S)$ is topologically generated by the set 
    \[
    \{\tau_1,\tau_2,A_1A'_6C_1B_3\overline{B_{11}}\mkern3mu\overline{C_{12}}\mkern3mu\overline{A'_8}\mkern3mu\overline{A_{13}}\}.
    \]
\end{theorem}

Since $\tau_1$ and $\tau_2$ are involutions, it remains to modify the last element to be an involution.

\begin{thmC}\label{thm:C}
     Let $S$ be the Jacob's Ladder surface. Then $\mod(S)$ is topologically generated by three involutions.
\end{thmC}
\begin{proof}
    Let 
    \[I = \{\tau_1,\tau_2,\tau_3A_1A'_6C_1B_3\overline{B_{11}}\mkern3mu\overline{C_{12}}\mkern3mu\overline{A'_8}\mkern3mu\overline{A_{13}}\},
    \]
    where $\tau_3 = H^6\tau_2\overline{H^6}$, and $G$ be the subgroup of $\mod(S)$ topologically generated by $I$. It is clear $\tau_3$ is in $G$, thus \[
    A_1A'_6C_1B_3\overline{B_{11}}\mkern3mu\overline{C_{12}}\mkern3mu\overline{A'_8}\mkern3mu\overline{A_{13}} \in G.
    \] By Theorem~\ref{lem:jaclem}, this implies that $G$ is $\mod(S)$.
    It is easy to check that \[
    \tau_3A_1A'_6C_1B_3\tau_3 = A_{13}A'_8C_{12}B_{11},
    \] 
    
    which implies by Lemma~\ref{lemm8} that $\tau_3A_1A'_6C_1B_3\overline{B_{11}}\mkern3mu\overline{C_{12}}\mkern3mu\overline{A'_8}\mkern3mu\overline{A_{13}}$ is an involution and we are done.
\end{proof}

\subsection{The Loch Ness Monster Surface}

Throughout this subsection, the Loch Ness Monster surface is denoted by $S$. Consider the rotations $\tau_1$ and $\tau_2$ defined as the rotation by $\pi$ radians as shown in Figures~\ref{fig:tau1n1} and \ref{fig:tau2n1}. It is clear that the product $\tau_1\tau_2$ is a handle shift, which we will call $H$, whose attracting end and repelling end are the same. One can think of the genera as lying along an infinite chain, bent so that both ends extend in the same direction; the map $H$ then acts by shifting each link (genus) along this chain by one.   

\begin{figure}[htbp]
      \centering
      \includegraphics[width=0.55\textwidth]{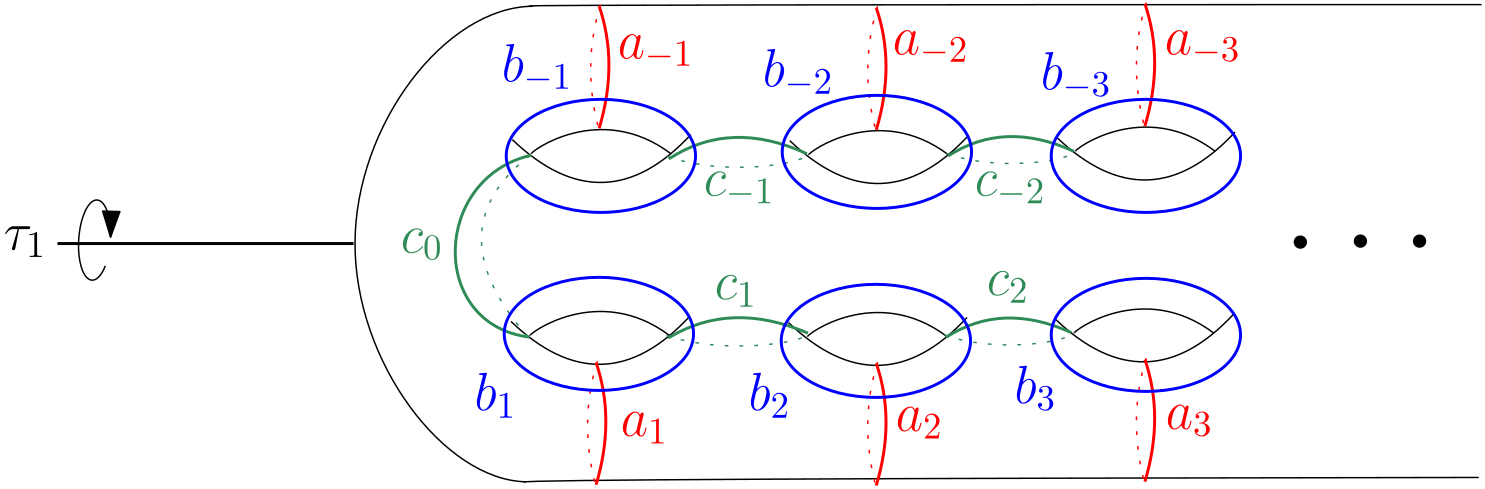}
      \caption{The rotation $\tau_1$.}
      \label{fig:tau1n1}
   \end{figure}
   
\begin{figure}[htbp]
      \centering
      \includegraphics[width=0.55\textwidth]{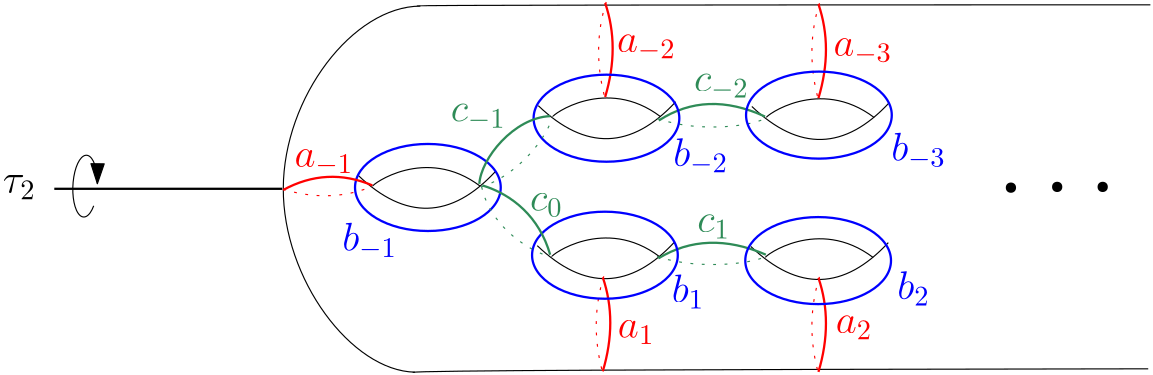}
      \caption{The rotation $\tau_2$.}
      \label{fig:tau2n1}
\end{figure}
   
The following lemma is analogous to Theorem~\ref{thm:genlemma}, and gives us a nice finite generating set for the Dehn twists required to topologically generate $\mod(S)$.

\begin{lemma}[{\cite[Lemma 3.12]{mingen}}]\label{lem:a1a2}
    The subgroup of $\mod(S)$ generated by 
    \[\{H,A_1\overline{A_2},B_1\overline{B_2},C_1\overline{C_2} \}\]
    contains the Dehn twists $A_i,B_i$ and $C_j$ for all $|i| \geq 1$ and $j\in \mathbb{Z}$.
\end{lemma}

\begin{thmD}\label{thm:D}
    The mapping class group of the Loch Ness Monster surface, $\mod(S)$, can be topologically generated by three involutions.
\end{thmD}
\begin{proof}
    Let $G$ be the subgroup topologically generated by the set
    \[
\set{\tau_1, \tau_2,\tau_2A_{-2}B_{-3}C_{-4}\overline{C_3}\mkern3mu\overline{B_2}\mkern3mu\overline{A_1}}.
    \]
    Observe that $H=\tau_1\tau_2$ is in $G$ Let
    \begin{align*}
    F_1=A_{-2}B_{-3}C_{-4}\overline{C_3}\mkern3mu\overline{B_2}\mkern3mu\overline{A_1}.
    \end{align*}
    Conjugating $F_1$ by $H^2$, we get
    \begin{align*}
        F_2=F_1^{H^2} = A_1B_{-1}C_{-2}\overline{C_5B_4A_3}\in G.
    \end{align*}
    Then, since $B_{-3}$ intersects once with $C_{-2}$ and $\overline{B_4}$ intersects once with $\overline{C_3}$, by the braid relation,
    \begin{align*}
        F_3 = F_2^{F_2F_1}=A_1B_{-1}B_{-3}\overline{C_5}\mkern3mu\overline{C_3}\mkern3mu\overline{A_3}\in G.
    \end{align*}
    Next,
    \begin{align*}
        &F_4 = F_2\overline{F_3} = C_{-2}\overline{B_{-3}}\mkern3mu\overline{B_4}C_3 \in G.
    \end{align*}
    Using $F_4$ and the handle shift $H$, we get
    \begin{align*}
        &\overline{F_4}^{\overline{H}}=\overline{C_2}B_3B_{-4}\overline{C_{-3}}\in G,\\
        & F_4^{\overline{H}^2} = C_{-4}\overline{B_{-5}}\mkern3mu\overline{B_2}C_1\in G.
    \end{align*}
    By conjugating $F_3$ by $H^5$, we get that
    \begin{align*}
        F_5 = \overline{F_3}^{H^5} = A_8C_8C_{10}\overline{B_3}\mkern3mu\overline{B_5}\mkern3mu\overline{A_6} \in G.
    \end{align*}
    Since $\overline{B_3}$ intersects once with $\overline{C_2}$ and $\overline{C_2}$ intersects once with $\overline{B_2}$, it follows by the braid relation that
    \begin{align*}
        &F_6 = F_5^{F_5\overline{F_4}^{\overline{H}}}=A_8C_8C_{10}\overline{C_2}\mkern3mu\overline{B_5}\mkern3mu\overline{A_6}\in G,\\
        & F_5\overline{F_6}=\overline{B_3}C_2 \in G,\\
        &F_7 = F_6^{F_6F_4^{\overline{H}^2}}=A_8C_8C_{10}\overline{B_2}\mkern3mu\overline{B_5}\mkern3mu\overline{A_6}\in G,\\
        &F_6\overline{F_7} = \overline{C_2}B_2 \in G.
    \end{align*}
    Hence,
    \begin{align*}
        &(\overline{B_3}C_2)(\overline{C_2}B_2)=\overline{B_3}B_2=B_2\overline{B_3} \in G.\\
        &(B_2\overline{B_3})^{\overline{H}}=B_1\overline{B_2} \in G.
    \end{align*}
    Moreover,
    \[
    (\overline{C_2}B_2)\conj{(\overline{B_3}C_2)}{\overline{H}} = (\overline{C_2}B_2)(\overline{B_2}C_1) =C_1\overline{C_2} \in G.
    \]
    Since $B_1\overline{B_2}$ and $C_1\overline{C_2}$ are both in $G$,
    \begin{align*}
    (B_{-3}\overline{B_{-2}})(B_{-2}\overline{B_{-1}})(B_{-1}\overline{B_1})(B_{1}\overline{B_2})=B_{-3}\overline{B_2} \in G,
    \end{align*}
    and
    \begin{align*}
    (C_{-4}\overline{C_{-3}})\dots (C_{1}\overline{C_2}) (C_2\overline{C_3})=C_{-4}\overline{C_3} \in G.
    \end{align*}
    Then,
    \begin{align*}
    (B_2\overline{B_{-3}})F_1(C_3\overline{C_{-4}})=A_{-2}\overline{A_1} \in G.
    \end{align*}
    By a similar argument, $B_{-2}\overline{B_2}$ is also in $G$, thus $\overline{A_{-2}\overline{A_1}}=A_1\overline{A_{-2}}$ is in $G$. Then, using these, we get
    \begin{align*}
        (A_{-2}\overline{A_1})^{A_{-2}\overline{A_1}B_{-2}\overline{B_2}}=B_{-2}\overline{A_1} \in G.
    \end{align*}
    Using $(A_{-2}\overline{A_1})^H=A_{-1}\overline{A_2}$, we get
    \begin{align*}
        (B_{-1}\overline{B_{-2}})(B_{-2}\overline{A_1})=B_{-1}\overline{A_1} \in G,\\(B_{-1}\overline{A_1})^{B_{-1}\overline{A_1}A_{-1}\overline{A_2}}=A_{-1}\overline{A_1} \in G,\\
        (A_{-1}\overline{A_1})^H=A_1\overline{A_2} \in G.
    \end{align*}
    By the Lemma~\ref{lem:a1a2}, G contains the Dehn twists $A_i,B_i$ and $C_j$ for all $|i| \geq 1$ and $j\in \mathbb{Z}$. It follows from [\citenum{tez}, Proposition 6.1.15] that 
    \[
    G \cong \mod(S).
    \]
    It remains to check that $\tau_2A_{-2}B_{-3}C_{-4}\overline{C_3}\mkern3mu\overline{B_2}\mkern3mu\overline{A_1}$ is an involution, which follows from Lemma~\ref{lemm8} and the fact that
    \[
    \tau_2A_{-2}B_{-3}C_{-4}\tau_2 = A_1B_2C_3.
    \]
\end{proof}

\section{Torsion Generators}

In this section, we give torsion generators for $n\geq8$. We use the same methods as in the involution generators case. To do so, we use a slightly modified version of Theorem~\ref{thm:genlemma}. Unlike the involution case, we cannot write a handle shift as a composition of elements of order $n$, and therefore need to add a handle shift factor in one of the generators. Since we obtain torsion elements using Lemma~\ref{lemm8}, handle shifts must come in pairs, and thus we need to modify Theorem~\ref{thm:genlemma}.

\begin{theorem}[{\cite[Theorem~3.5]{mingen}}]\label{thm:genthm}
For $n \geq 3$, the group generated by the set
\[
    \set{R, A_1^1\overline{A_1^2}, B_1^1\overline{B_1^2}, C_0^1\overline{C_0^2}, h_{1,2}}
\]
contains Dehn twists $A_i^j$, $B_i^j$, $C_{i-1}^j$ for all $i \geq 1$ and $1\leq j \leq n$.
\end{theorem}

We now give our torsion generators.

\begin{thmE}\label{thm:E}
    For even $n \geq 8$, $\mod(S(n))$ is topologically generated by four torsion elements of order $n$.
\end{thmE}
\begin{proof}
Let $R,R',T$ be homeomorphisms of $S(n)$ that acts on the space of ends as the cycles $(1\, \dots \, n)$, $(1\,2\,n \,n-1 \, n-2 \, \dots \, 3)$, and $(1\,7\,2\,5\,3\,6\,4\, 9 \, 10 \, \dots \, n-1 \, n\, 8)$ (which reduces to $(1\,7\,2\,5\,3\,6\,4\,8)$ when $n=8$), respectively. These homeomorphisms can be realized as rotations on certain embeddings of $S(n)$, see Figure~\ref{fig:r't}.
    
\begin{figure}[htbp]
    \centering
    \begin{minipage}[b]{0.75\textwidth}  
        \centering
        \adjustbox{valign=m}{\includegraphics[width=\linewidth, height=6cm, keepaspectratio]{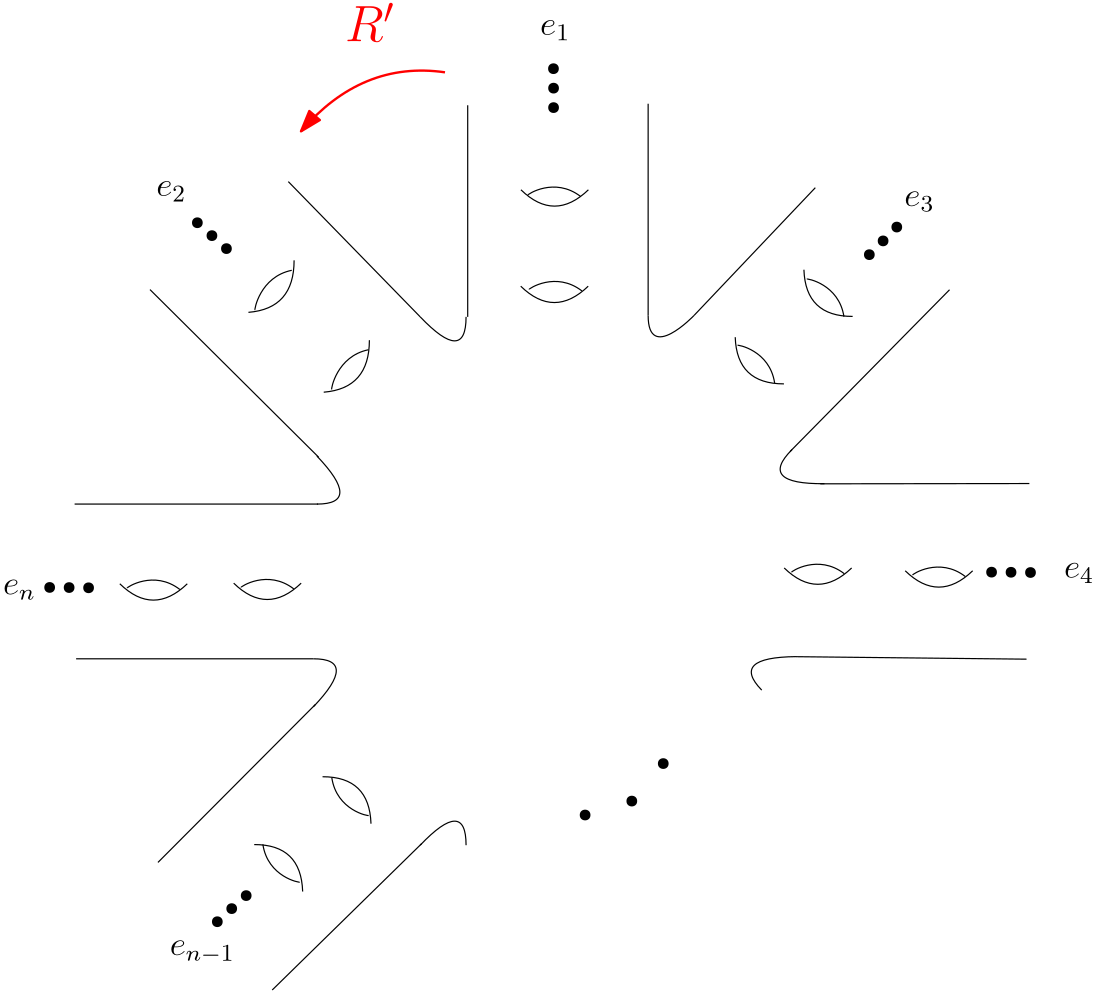}}
    \end{minipage}
    \hfill 
    \begin{minipage}[b]{0.75\textwidth}  
        \centering
        \adjustbox{valign=m}{\includegraphics[width=\linewidth, height=6cm, keepaspectratio]{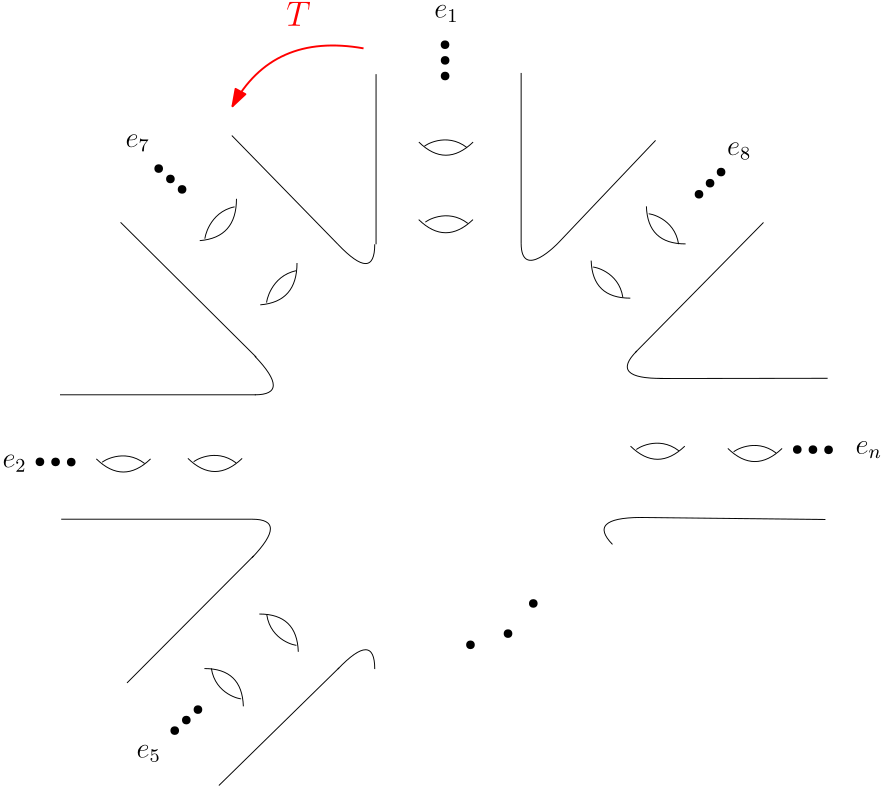}}
    \end{minipage}
    \caption{The rotations $R'$ and $T$ on $S(n)$ with even $n$}
    \label{fig:r't}
\end{figure}
    
    Let $G$ be the subgroup of $\mod(S(n))$ topologically generated by the set
    \[
    \set{R,R',T,RF_1},
    \]
    where $F_1=B^1_1A^1_2C^1_3\overline{B^2_1}\mkern3mu\overline{A^2_2} \mkern3mu\overline{C^2_3}h_{n-2,n-1}\overline{h_{n-1,n}}$.
    Since $R$ is in $G$, $F_1$ is also in $G$. Conjugating $F_1$ by $R^3$, we obtain
    \[
    F_2= \conj{F_1}{R^3}=B^4_1A^4_2C^4_3\overline{B^5_1}\mkern3mu\overline{A^5_2} \mkern3mu\overline{C^5_3}h_{1,2}\overline{h_{2,3}} \in G.
    \]
    Observe that the curves related to the Dehn twist factors of $F_1$ and $F_2$ do not intersect. Since $n\geq 8$, the handle shift factors in these elements also commute with the Dehn twists in the other (which is in fact the reason we require that $n\geq8$). It follows that 
    
    \begin{align*}
    F_3=\conj{F_1}{F_1F_2} &= \conj{F_1}{h_{1,2}\overline{h_{2,3}}} \\
                       &= (h_{1,2}\overline{h_{2,3}})F_1(h_{2,3}\overline{h_{1,2}}) \\
                       &= B^2_1A^1_1C^1_2\overline{B^2_3}\mkern3mu\overline{A^2_4} \mkern3mu\overline{C^2_5}h_{n-2,n-1}\overline{h_{n-1,n}} \in G.
    \end{align*}
    Then,
    \[
    F_4 = \conj{F_3}{F_3F_1} = B^2_1B^1_1C^1_2\overline{C_3^2}\mkern3mu\overline{A_4^2}\mkern3mu\overline{C_5^2}h_{n-2,n-1}\overline{h_{n-1,n}} \in G,
    \]
    and 
    \[
    F_5= F_3\overline{F_4} = A^1_1\overline{B^2_3}C^2_3\overline{B^1_1} \in G.
    \]
    
    By conjugation,
    \begin{align*}
    F_6 = (\conj{F_5}{\overline{R}})^{F_2} &=\conj{(A^n_1\overline{B^1_3}C^1_3\overline{B^n_1})}{F_2}\\
    &= A^n_1\overline{B^1_2}C^1_2\overline{B^n_1} \in G,
    \end{align*}
    and 
    \[
    F_7 = \conj{F_5}{\overline{R}}\overline{F_6} = (A^n_1\overline{B^1_3}C^1_3\overline{B^n_1})(B^n_1\overline{C^1_2}B^1_2\overline{A^n_1}) = \overline{B^1_3}C^1_3\overline{C^1_2}B^1_2 \in G.
    \]
    Note that $\conj{F_7}{R} =\overline{B^2_3}C^2_3\overline{C^2_2}B^2_2$ is also in $G$.
    Through a series of conjugations,
    \begin{align*}
        & \conj{(\overline{B^2_3}C^2_3\overline{C_2^2}B^2_2)}{\overline{F_2}} = \overline{B^2_1}C^2_1\overline{C_0^2}B^3_1 \in G, \\
        & \conj{(\overline{B^2_1}C^2_1\overline{C_0^2}B^3_1)}{\overline{R}^2} = \overline{B^n_1}C^n_1\overline{C_0^n}B^1_1 \in G, \\
        & \conj{(\overline{B^n_1}C^n_1\overline{C_0^n}B^1_1)}{F_2} = \overline{B^n_1}C^n_1\overline{C_0^n}B^2_1 \in G, \\
        & \conj{(\overline{B^n_1}C^n_1\overline{C_0^n}B^2_1)}{R^2} = \overline{B^2_1}C^2_1\overline{C_0^2}B^4_1 \in G, \\
        & (\overline{B^2_1}C^2_1\overline{C_0^2}B^3_1)(\overline{\overline{B^2_1}C^2_1\overline{C_0^2}B^4_1}) = B^3_1\overline{B^4_1} \in G.
    \end{align*}
    It immediately follows from conjugation by $\overline{R}^2$ that $B^1_1\overline{B^2_1}$ is in $G$. Then,
    \[
    F_8 = \conj{(B^1_1\overline{B^2_1})}{\overline{F_2}} = B^1_2\overline{B^1_1}\in G.
    \]
    Then,
    \[
    \conj{(B^1_2\overline{B^1_1})}{(B^1_2\overline{B^1_1})F_1} = A_2^1\overline{B^1_1} \in G,
    \]
    and therefore,
    \[
    \conj{B_1^1\overline{A^1_2}}{(B_1^1\overline{A^1_2})F_3} = A^1_1\overline{A^1_2} \in G.
    \]
    Using these,
    \begin{align*}
        & (A^1_1\overline{A^1_2})(A^1_2\overline{B^1_1})(B^1_1\overline{B^2_1})\conj{(B^1_1\overline{A^1_2})}{R}\conj{(A^1_2\overline{A^1_1})}{R} = (A^1_1\overline{A^1_2})(A^1_2\overline{B^1_1})(B^1_1\overline{B^2_1})(B^2_1\overline{A^2_2})(A^2_2\overline{A^2_1}) =   A^1_1\overline{A^2_1} \in G.
    \end{align*}

   The fact $B^1_1\overline{B^2_1}$ is in $G$ implies that 
   \[
   (B^2_1\overline{B^3_1})(B^3_1\overline{B^4_1}) =B^2_1\overline{B^4_1} \in G,
   \]
   and since $\overline{B^2_1}C^2_1\overline{C_0^2}B^4_1$ is in $G$, and the curve $b^4_1$ is disjoint from $c^2_1$ and $c_0^2$,
   \[
   (B^2_1\overline{B^4_1})(\overline{B^2_1}C^2_1\overline{C_0^2}B^4_1) = C^2_1\overline{C_0^2}
   \]
   is also in $G$. It follows that
   \begin{align*}
   & \conj{(C_0^2\overline{C^2_1})}{(C_0^2\overline{C^2_1})(B^3_1\overline{B^4_1})} = B^3_1\overline{C^2_1} \in G, \\
   & (C_0^2\overline{C^2_1})(C^2_1\overline{B^3_1}) = C_0^2\overline{B^3_1} \in G, \\
   & \conj{(C^2_0\overline{B^3_1})}{\overline{R}}(B^2_1\overline{B^3_1})(B^3_1\overline{C^2_0}) = (C^1_0\overline{B^2_1})(B^2_1\overline{B^3_1})(B^3_1\overline{C^2_0})=C^1_0\overline{C^2_0} \in G.
   \end{align*}

    We showed that $A_1^1\overline{A_1^2},B^1_1\overline{B_1^2}$ and $C_0^1\overline{C_0^2}$ are in $G$.
    
    To use Theorem~\ref{thm:genthm}, we need to show that $h_{1,2}$ is in $G$. We begin showing this by removing the Dehn twist factors of $F_1$ in pairs. Since $B^1_1\overline{B_1^2}$ is in $G$,
    \[
    F_8 = (B_1^2\overline{B^1_1})F_1 = A^1_2C^1_3\overline{A^2_2} \mkern3mu\overline{C^2_3}h_{n-2,n-1}\overline{h_{n-1,n}} \in G.
    \]
    Note that
    \begin{align*}
        &\conj{(A_1^1\overline{A_1^2})}{\overline{F_2}}=A_2^1\overline{{{A'}_1^1}} \in G,\\
        &(\conj{(A_1^1\overline{A^2_1})}{R})^{\overline{F_2}} = \conj{(A_1^2\overline{A^3_1})}{\overline{F_2}} = {A'}_1^1\overline{A^3_2} \in G,  \\
        & \conj{((A_2^1\overline{{{A'}_1^1}})({A'}_1^1\overline{A^3_2}))}{\overline{R^2}} = \conj{(A_2^1\overline{A^3_2})}{\overline{R^2}} = A_2^{n-1}\overline{A^1_2} \in G,\\
        &(A_2^{n-1}\overline{A^1_2})(A_2^1\overline{{{A'}_1^1}})=A_2^{n-1}\overline{{{A'}_1^1}} \in G,\\
        &\conj{(A_2^{n-1}\overline{{{A'}_1^1}})}{F_2}=A_2^{n-1}\overline{A_1^2} \in G,\\
        &(A_2^{n-1}\overline{A_1^2})(A_1^2\overline{A_1^1})\conj{(A_1^2\overline{A^{n-1}_2})}{\overline{R}}= (A_2^{n-1}\overline{A_1^2})(A_1^2\overline{A_1^1})(A_1^1\overline{A_2^{n-2}}) = A_2^{n-1}\overline{A_2^{n-2}} \in G,\\
        &\conj{(A_2^{n-1}\overline{A_2^{n-2}})}{R^3}= A_2^2\overline{A_2^1} \in G.
    \end{align*}
    Therefore,
    \[
    F_9 = (A_2^2\overline{A_2^1})F_8 = C^1_3\overline{C^2_3}h_{n-2,n-1}\overline{h_{n-1,n}} \in G.
    \]
    Finally, the fact that
       \begin{align*}
        &(\conj{(C_1^1\overline{B_1^1})}{\overline{F_2}})^{\overline{F_2}} = C_3^1\overline{B_3^1} \in G,\\
        &(\conj{(B_1^1\overline{B_1^2})}{\overline{F_2}})^{\overline{F_2}} = \conj{(B_2^1\overline{B_1^1})}{\overline{F_2}} =B_3^1\overline{B_2^1} \in G,\\
        &(C_3^1\overline{B_3^1})(B_3^1\overline{B_2^1}) = C_3^1\overline{B_2^1} \in G, \\
        &\conj{(C_3^1\overline{B_2^1})}{R}\conj{(B_2^1\overline{B_1^1})}{R}(B_1^2\overline{B^1_1})= (C_3^2\overline{B_2^2})(B_2^2\overline{B_1^2})(B_1^2\overline{B_1^1})=  C_3^2\overline{B_1^1} \in G,\\
        & (C_3^2\overline{B_1^1})(B^1_1\overline{B_2^1})(B^1_2\overline{C_3^1}) = C_3^2\overline{C_3^1} \in G.
       \end{align*}
   implies that,
   \[
   (C_3^2\overline{C_3^1})F_9 = h_{n-2,n-1}\overline{h_{n-1,n}} \in G.
   \]
    It immediately follows from conjugation by $R^3$ that $h_{1,2}\overline{h_{2,3}}$ is in G. It is clear that $h_{2,3}\overline{h_{3,4}}$ and $h_{3,4}\overline{h_{4,5}}$ are also in $G$. Thus,
    \[
    h_{2,3}\overline{h_{3,4}}h_{3,4}\overline{h_{4,5}} = h_{2,3}\overline{h_{4,5}} \in G.
    \]
    By the same argument, $h_{5,6}\overline{h_{7,8}}$ is also in $G$. Observe that
    \[
    \conj{(h_{5,6}\overline{h_{7,8}})}{T} = h_{3,4}h_{1,2} 
    \]
    is in $G$. It immediately follows that
    \[
    h_{2,3}\overline{h_{3,4}}h_{3,4}h_{1,2}  = h_{1,3} \in G.
    \]
    Finally,
    \[
    \conj{(h_{1,3})}{T} = T h_{1,3} \overline{T} = h_{7,6} = \overline{h_{6,7}} \in G,
    \]
    and conjugation by powers of $R$ shows that $h_{i,i+1}$ is in $G$ for all $1 \leq i \leq n-1$. By Theorem~\ref{thm:genthm}, $G$ contains the Dehn twists $A_i^j$, $B_i^j$, $C_{i-1}^j$ for all $i\geq 1$ and $1\leq j\leq n$, and by [\citenum{tez}, Proposition 6.1.10], $\overline{\pmod_\mathrm{c}(S(n))}$ is a subgroup of $G$. Since $G$ also contains $h_{i,i+1}$ for all $1 \leq i \leq n-1$, it contains $\pmod(S(n))$.
    
To show that $G$ is the entire $\mod(S(n))$, we need to show that it contains mapping classes whose images in $\Sym_n$ generate it. For brevity, we shall conflate $G$ with its image inside $\Sym_n$.
    Clearly,
    \[
    (1\, \dots \, n)(1\,2\,n \,n-1 \, n-2 \, \dots \, 3) = (1\, 3 \,2),
    \]
    and
    \[
    (1\, \dots \, n)^k(1\, 2\, 3)\overline{(1\,2\, \dots \, n)^k} = (k+1\ k+2 \ k+3),
    \] 
    for $1\leq k \leq n-1$. It follows that $(1\,2\,3), (2\,3\,4), \dots ,(n\,1\,2)$ are all in $G$. Observe that 
    \[
    (k \ k+1 \ k+2)(1\,2 \, k)\overline{(k \ k+1 \ k+2)} = (1\,2 \ k+1)
    \]
    for $3\leq k \leq n-2$. Therefore, $(1\,2\,4), (1\,2\,5),\dots, (1\,2\ n-1)$ are all in $G$. Since,
    \[
    \overline{(1\, 2\, k)}  = (1\, k\, 2),
    \]
    and
    \[
    (1\, 2\, k)(1\, j\, 2) = (1 \, j \, k),
    \]
    every $3$-cycle of the form $(1,k,j)$ is in $G$. For $i,j,k>1$, it is clear that
    \[
    (1 \, i \,j)(1\, j \, k) = (i \, j \, k),
    \]
    which implies that $G$ contains every $3$-cycle. Finally,
    \[
    (1\,2\,4)(1\,4\,5)(1\,5\,6) \dots (1 \ n-1 \ n) = (1\,2\,4\,\dots\,n) \in G,
    \]
    and 
    \[
    (1\,2\,3\,4\,\dots\,n)\overline{(1\,2\,4\,5\,\dots\,n)} = (3\,4) \in G.
    \]
    It follows from conjugation by $\overline{(1\,2\,3\,4\,\dots\,n)^2}$ that $(1\,2)$ is in $G$. Since $(1\,2\,3\,4\,\dots\,n)$ and $(1\,2)$ generate the symmetric group, $G$ is the entire $\mod(S(n))$.
    All that remains is to show that the order $RF_1$ is $n$. Note that,
    \[
    \conj{(B^1_1A^1_2C^1_3h_{n-2,n-1})}{R} = R(B^1_1A^1_2C^1_3h_{n-2,n-1}) \overline{R} =B^2_1A^2_2C^2_3h_{n-1,n}.
    \]
    By Lemma~\ref{lemm8}, $RF_1 = R(B^1_1A^1_2C^1_3h_{n-2,n-1})\overline{(B^2_1A^2_2C^2_3h_{n-1,n})}$ has order $n$, and we are done.
   
\end{proof}

For odd $n$, any $n$-cycle is an even permutation. Since the product of any two even permutations is even, two $n$-cycles cannot generate $\Sym_n$. Therefore for odd $n$, the element $R'$ cannot be an $n$-cycle if we wish to generate the entire mapping class group. Therefore, we compromise by putting an element with order $n-1$.

\begin{thmF}\label{thm:F}
    For odd $n \geq 8$, $\mod(S(n))$ is topologically generated by three torsion elements of order $n$ and one torsion element of order $n-1$.
\end{thmF}
\begin{proof}
    Let $G$ be the subgroup of $\mod(S(n))$ topologically generated by the set
    \[
    \set{R,R',T,RF_1},
    \]
    where $F_1=B^1_1A^1_2C^1_3\overline{B^2_1}\mkern3mu\overline{A^2_2} \mkern3mu\overline{C^2_3}h_{n-2,n-1}\overline{h_{n-1,n}}$, and $R'$ is a homeomorphism whose image in $\Sym_n$ is the $n-1$-cycle $(1\,2\,4\, \dots \, n)$. This homeomorphism can be realized geometrically as the $\dfrac{2\pi}{n-1}$ radian counterclockwise rotation of the model depicted in Figure~\ref{fig:n-1}. The proof is identical to that of Theorem~\ref{thm:E}, except for the fact that $R'$ directly projects to $(1\,2\,4\, \dots \, n)$, and therefore the derivation of it is skipped.
\begin{figure}[H]
      \centering
      \scalebox{0.25}{\includegraphics{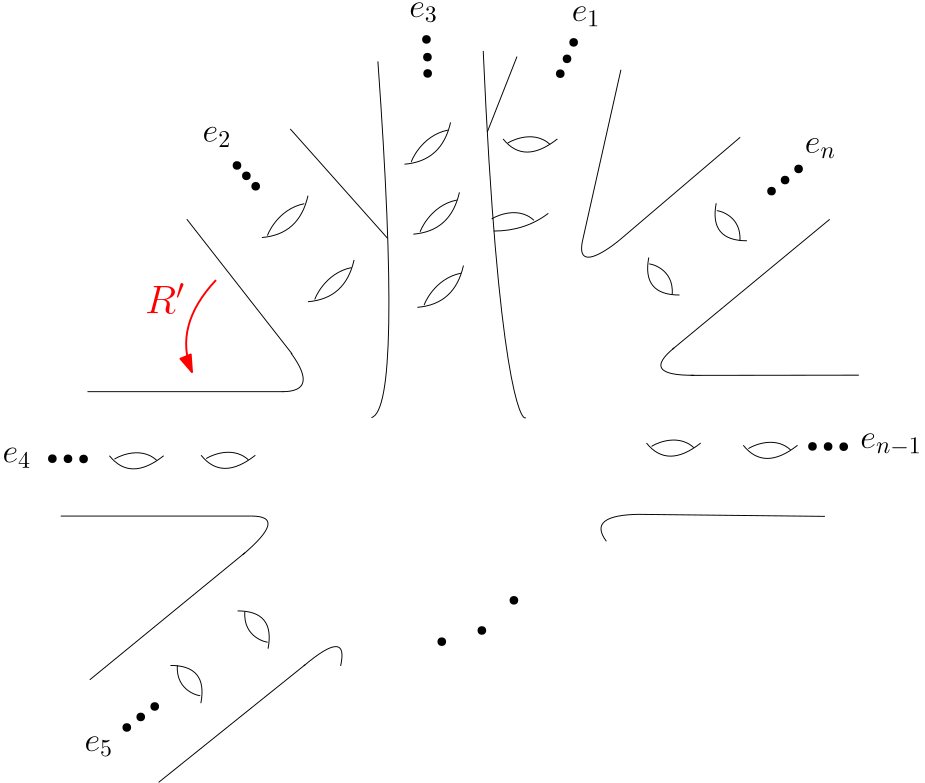}}
      \caption{The rotation $R'$ on $S(n)$ with odd $n$}
      \label{fig:n-1}
   \end{figure}    
\end{proof}

\bibliographystyle{amsplain}
\bibliography{references}

\end{document}